\documentclass[11pt]{article}
\usepackage{amsfonts}
\usepackage[margin=2cm]{geometry}
\usepackage{graphicx,color}
\usepackage[hang,small,bf]{caption}
\usepackage{xspace}
\usepackage{enumerate}
\usepackage{amssymb,amsmath}
\usepackage{amsthm}

\usepackage[T1]{fontenc}
\usepackage[utf8]{inputenc}
\usepackage{authblk}

\newtheorem{thm}{Theorem}

\newtheorem{proposition}[thm]{Proposition}
\newtheorem{lemma}[thm]{Lemma}

\newtheorem{definition}[thm]{Definition}

\begin{document}

\title{The solution to an open problem for a caching game}

\author[*]{Endre Cs\'{o}ka}
\author[**]{Thomas Lidbetter}
\affil[*]{\emph{Mathematics Institute, University of Warwick, Coventry, CV4 7AL, csokaendre@gmail.com}}
\affil[**]{\emph{Department of Mathematics, London School of Economics, Houghton Street, London WC2A 2AE, t.r.lidbetter@lse.ac.uk}}

 \date{}
  \maketitle
  \thispagestyle{empty}
\maketitle

 \begin{abstract}
In a caching game introduced by Alpern \emph{et al.} \cite{AFOL}, a Hider who can dig to a total fixed depth normalized to $1$ buries a fixed number of objects among $n$ discrete locations. A Searcher who can dig to a total depth of $h$ searches the locations with the aim of finding all of the hidden objects. If he does so, he wins, otherwise the Hider wins. This zero-sum game is complicated to analyze even for small values of its parameters, and for the case of $2$ hidden objects has been completely solved only when the game is played in up to $3$ locations. For some values of $h$ the solution of the game with $2$ objects hidden in $4$ locations is known, but the solution in the remaining cases was an open question recently highlighted by Fokkink \emph{et al.} \cite{Fokkink-et-al}. Here we solve the remaining cases of the game with $2$ objects hidden in $4$ locations. We also give some more general results for the game, in particular using a geometrical argument to show that when there are $2$ objects hidden in $n$ locations and $n \rightarrow \infty$, the value of the game is asymptotically equal to $h/n$ for $h \ge n/2$.

 \end{abstract}{\bf Keywords:} search, caching, zero-sum game, accumulation game

\newpage
\setcounter{page}{1}

\section{Introduction}
In \cite{AFOL}, Alpern \emph{et al}. introduced a new type of search game called a \emph{caching game}. A caching game is a zero-sum game between a Searcher and a Hider in which the Hider has some material that he wishes to hide or \emph{cache} in several possible hiding places. The Hider could be a terrorist caching weapons or explosives, or, as in \cite{AFLC}, the Hider could be an animal such as a squirrel, caching nuts or some other food. The Hider's aim is to end up with a certain minimal amount of material. In the case of a terrorist, this may be a minimal amount required to carry out an attack, or in the case of a squirrel it could be a minimal amount of food to survive the winter. The Searcher, who has a limited amount of resources with which to search, has to decide how to distribute these resources about the hiding locations in order to maximize the probability that the Hider will be left with insufficient material.

The caching game we discuss in this paper (defined formally later in Section \ref{sec:main}) takes place in a finite number of locations and the Hider's material takes the form of a finite number of objects, which he can ``bury'' in the locations. He is limited by the total amount that he can ``dig'', but he is permitted to bury multiple objects in the same location. We shall see later that this allows him to mislead the Searcher by using ``decoys''. The Hider's limitation could correspond to a time restriction in the case of a terrorist caching weapons, or an energy restriction in the case of a squirrel burying nuts. The Searcher also has a restriction on the amount he is able to dig. The Searcher can use an \emph{adaptive} strategy (called a \emph{smart} strategy in \cite{AFLC}), meaning that he is not required to specify the depths to which he will dig in each location in advance, but can change his plan based on new information discovered in the course of the search. We model this as a zero-sum win-lose game, where the Hider wins if and only if he is left with at least one object after the Searcher has finished searching.

In general, computing the value of the game and the optimal mixed strategies seems to be hard, not least because both players have infinite strategy sets. In \cite{AFOL}, the game was solved in the case that the Hider hides $2$ objects in $2$ locations and the case that he hides $2$ objects in $3$ locations. For certain values of the energy parameters, Fokkink \emph{et al}. \cite{Fokkink-et-al} describe the solution to the game for $2$ objects in $4$ locations, taken from \cite{OpDenKelder}. In Section \ref{sec:main} of this work we solve the game for $2$ objects in $4$ locations in the previously unsolved cases. We also give some more general results for the game played in an arbitrary number of locations in Section \ref{sec:general}.

\section{Related literature}

Caching games are a natural ancestors of \emph{accumulation games}, introduced by Kikuta and Ruckle (\cite{Kikuta-Ruckle97}, \cite{Kikuta-Ruckle2000} and \cite{Kikuta-Ruckle2002}) and further studied in \cite{Alpern-Fokkink}, \cite{AFK} and \cite{AFP}. Accumulation games, in their most general form, take place between a Hider who accumulates resources in stages over several time periods, and a Searcher who confiscates some of the resources in every period. Caching games are single stage versions of accumulation games, but in the caching game we study here, the amount of material the Searcher can confiscate from a location depends upon the amount of his energy that he dedicates to this locations. This adds an extra layer of complexity on top of traditional accumulation games in which, upon specifying some subset of locations, the Searcher confiscates all the material cached at those locations. The first author also extends caching games by considering ``limit games'' \cite{Csoka} in which the set of locations of the hidden objects is not a finite set but an interval $[0,T]$ or $[0,\infty)$.

Accumulation games and caching games are also related to \emph{inspection games}, as studied, for example in \cite{Avenhaus-et-al}, \cite{Baston-Bostock} and \cite{Ferguson-Melolikdakis}. Inspection games model a situation in which an \emph{inspectee} is legally obliged to comply with some regulation such as an arms control treaty, and an \emph{inspector} wishes to detect a violation. Caching games also have a strong geometric flavor and are related to geometric games of search and ambush, as found in \cite{Garnaev}, \cite{Ruckle} and \cite{Zoroa-et-al}. See \cite{Baston} for some interesting open problems in the area.

This work sits more generally in the field of \emph{search games}, good accounts of which can be found in the monographs \cite{Alpern-Gal} and \cite{Garnaev-book}; and \emph{search theory}, as surveyed in \cite{Benkosi-et-al}.

\section{The model and main results}
\label{sec:main}

In this section we describe the game, as defined in \cite{AFOL}, and give the solution for $2$ objects hidden in $4$ locations in the previously unsolved cases.

\subsection{Game definition and example}

The general form of the game that we consider takes place in $n$ locations. The Hider must choose where to bury $k$ objects in these locations and a strategy for him corresponds to $n$ sets, $(S_1,\ldots,S_n)$ with $\sum_{i \le n}\left\vert{S_i}\right\vert = k$, where each $S_i$ contains the (non-negative) depths at which objects are buried in location $i$. For ``location $i$'' we use the abbreviation $L_i$. The Hider is permitted to hide several objects in the same location, but he is restricted in the total amount he can dig, so that $\sum_{i=1}^n \max\{x:x \in S_i\}$ is no greater than a constant which we normalize to 1. In other words, the sum of the depths at which he buries the deepest object in each location is no greater than $1$. For convenience, if $S_i = \emptyset$ we write $S_i = 0$ and if $S_i=\{d\}$ is a singleton then we write $S_i=d$. For example, if $n=3,k=3$ and there are objects buried at depths $1/2$ and $2/3$ in $L_1$ and an object buried at depth $1/3$ in $L_2$ we write the strategy as $(\{1/2,2/3\},1/3,0)$. 

The Searcher's strategies are more complicated to describe, and we refer to \cite{AFOL} for a precise description. Informally, a Searcher strategy is a plan of how to dig in the locations. He can dig a total distance bounded above by some constant $h \ge 0$, and he can change his digging plan dynamically as he gains new information about the locations of the objects. The Searcher is permitted to dig simultaneously in more than one location (this may be seen as the limit of a strategy in which he alternates between locations, digging a very small amount $\epsilon$ in each one).

The payoff of the game is $1$ if the Searcher discovers all the objects; in this case we say the Searcher wins and the Hider loses. Otherwise, if the Hider is left with at least one object, the payoff is $0$, the Hider wins and the Searcher loses. We assume that $1 \le h < n$, otherwise the solution of the game is trivial. In \cite{AFOL} it was shown that the game has a near value and near optimal strategies (though in previously solved examples of the game, exact optimal strategies are found).

We illustrate the game with a simple example, as described in \cite{AFOL}, in the case of $n=2$ locations and $k=2$ objects. Observe that the Hider can always ensure that the value is at most $1/2$ by choosing randomly between $(1,0)$ and $(0,1)$, since the Searcher cannot dig to a depth of $1$ in both locations. If $h \ge 3/2$, then the Searcher can ensure an expected payoff of at least $1/2$ by digging to depth $1$ in a randomly chosen location and digging to depth $1/2$ in the other location. To see this, observe that if the objects are the in same location, the Searcher wins if he chooses to dig to depth $1$ in this location (which happens with probability $1/2$); if the objects are in different locations, we can assume without loss of generality that they are at depths $x$ in $L_1$ and $y$ in $L_2$ with $y \le 1/2$, in which case the Searcher wins if he digs to a depth of $1$ in $L_1$ (which happens with probability $1/2$). Hence the value of the game is $1/2$ if $h \in [3/2,2)$.

Now consider the same game with $2$ objects hidden in $2$ locations but with $h < 3/2$. In this case, the Hider can guarantee he wins with probability at least $1/3$ by choosing equiprobably between $(1/2,1/2)$, $(1/2,1)$ and $(1,1/2)$. It is clear that the Searcher cannot win against more than $1$ of these strategies. On the other hand, the Searcher can guarantee a win with probability at least $1/3$ by guessing equiprobably between the three possibilities that (i) the objects are in different locations, (ii) they are both in $L_1$ or (iii) they are both in $L_2$. If he guesses correctly, it is easy to see he only needs to dig a total depth of $1$ to win with certainty. Hence the value of the game is $1/3$ if $h \in [1,3/2)$. Note that this strategy relies on the Searcher's ability to adapt his search as he goes along.

We also remark that for the second Hider strategy described above (and trivially the first), after finding the first object, the second object will be optimally placed in the sub-game faced by the Searcher. Fokkink \emph{et al}. \cite{Fokkink-et-al} conjecture that the optimal Hider strategy always has this property (``A Kikuta-Ruckle Conjecture for Caching Games''). Indeed, this property can be seen in the solutions we present in this section, but in \cite{Csoka}, the first author shows the conjecture is not true in general.

The reader will notice that in the proof of the following lemmas, we do not make the assumption that the total depth dug by the Hider is precisely $1$. Counterintuitive as it may seem, it is shown in \cite{Csoka} that for some large values of $n$ the Hider's optimal strategy must with positive probability place the objects at depths whose sum is strictly less than $1$, so we must not neglect these strategies.

\subsection{New results for $2$ objects hidden in $4$ locations}

In \cite{AFOL}, the full solution of the game can be found for $k=2$ and $n \le 3$. In \cite{Fokkink-et-al}, the solution of the game for $n=4$ and $k=2$ is given for all values of $h$ except $h \in [7/4,2)$ and $h \in [11/5,7/3)$. Here we give solutions for these missing cases. The value of the game is the same for all $h \in [11/5,7/3)$, whereas for $h \in [7/4,2)$ the optimal strategies and value depend on which of three sub-intervals $h$ belongs to. Therefore our solution breaks down into four cases. 

We begin by summarizing the results in for the case $n=2,k=4$ in Theorem \ref{thm:summary}, emphasizing new results in bold.

\begin{thm}
\label{thm:summary}
The solution of the game for $2$ objects hidden in $4$ locations is shown in Table \ref{tab:main-summary}.
\begin{table}[ht]
\centering
\begin{tabular}{|c|c|}
\hline
$h$ & Value \\ \hline
$[1,3/2)$ & $1/10$ \\ 
$[3/2,5/3)$ & $3/20$ \\ 
$[5/3,7/4)$ & $1/5$ \\ 
$\textbf{[7/4,9/5)}$ & $\textbf{9/40}$ \\
$\textbf{[9/5,11/6)}$ & $\textbf{7/30}$ \\
$\textbf{[11/6,2)}$ & $\textbf{1/4}$ \\ 
$[2,11/5)$ & $2/5$ \\ 
$\textbf{[11/5,7/3)}$ & $\textbf{9/20}$ \\ 
$[7/3,3)$ & $1/2$ \\ 
$[3,4)$ & $3/4$ \\
\hline
\end{tabular}
\caption{Value of the game for $2$ objects in $4$ locations.}
	\label{tab:main-summary}
\end{table}
\end{thm}

We will prove the theorem using $4$ lemmas, corresponding to the $4$ cases, which we will present in increasing order of complexity, beginning with the simplest case, $h \in [11/6,2)$. The proofs are often very similar to one another and involve a systematic checking of cases, in which case we give a detailed exposition the first time and simply sketch the details in later proofs. The players' optimal strategies are of more interest than the proofs themselves. 

We first spend some time discussing how we will describe the optimal strategies for the players, as this is non-trivial. Since we are restricting our attention to the game with $2$ hidden objects, we can describe the Searcher's strategy in two stages, which we will call Stage 1 and Stage 2. Stage 1 specifies what the Searcher does until he finds the first object and Stage 2 describes what he does after finding the first object (if he does). 

To describe Stage 1 we simply give a sequence of vectors of the form $(x_1,x_2,x_3,x_4)$ which signify that at time $x_1+x_2+x_3+x_4$ the Searcher has dug to depth $x_i$ in $L_i$ for $i=1,\ldots,4$. At intermediate times he transforms from one vector to the next in a linear way. For example, the interpretation of the sequence $(1,0,0,0),(1,1/2,0,0),(1,1,1/2,0)$ would be to dig to depth $1$ in $L_1$, then dig to depth $1/2$ in $L_2$, and finally to dig simultaneously to depth $1$ in $L_2$ and to depth $1/2$ in $L_3$.

To describe Stage 2 we will specify some rules telling the Searcher the order in which to proceed to search for the second object after finding the first. Since the depth of the first object gives a natural restriction on the depth of the second object it will be clear how deep to dig in the other locations once the order is established. For example, if the rule is to search for the second object in $L_1$ and then $L_2$ after finding the first object then supposing the first object is found at depth $1/3$ in $L_1$, the Searcher should proceed to dig to depth $1$ in $L_1$ and then depth $2/3$ in $L_2$. We call a Stage 2 search of this type an {\em intelligent search}, abbreviated $IS$, and we denote Stage 2 of the Searcher's strategy in the preceding example by $IS(12)$. Similarly, for any sequence $\sigma$ of locations, we write $IS(\sigma)$ for the intelligent search in Stage 2 which searches for the second object in the locations specified by their order in $\sigma$ at the maximum possible depths they could be in those locations.

To denote the Hider's optimal strategies, we will distinguish between strategies where the objects are in different locations and those where they are in the same locations. For $x \in [0,1]$, let $D(x)$ be the set of all Hider strategies where one object is hidden at depth $x$ in some location and the other object is at depth $1-x$ in some other location. (So if $x \neq 1/2$ then $D(x)$ has size $3 \times 4=12$, and if $x=1/2$ then $D(x)$ has size $12/2=6$.) Also let $E(x)$ be the set of the $4$ Hider strategies where both objects are hidden in the same location at depths $x$ and $1$.

We begin with the case $h \in [11/6,2)$. 

\begin{lemma}
\label{lemma:11/6-2} For $n=4,k=2$ and $h \in [11/6,2)$, the value of the game is $1/4$. The Hider's optimal strategy is to choose equiprobably from the $4$ pure strategies in $E(1)$. 

The Searcher's optimal strategy is to number the locations randomly and proceed as follows.
\begin{enumerate}
\item[Stage 1:] Dig according to the sequence $(1/2,0,0,0),(1/2,1/2,0,0),(1,1/2,0,0)$.
\item[Stage 2:] If an object is found in $L_1$, follow $IS(1234)$; if an object is found in $L_2$, follow $IS(234)$.
\end{enumerate}
\end{lemma}

\begin{proof}
First note that the given Hider strategy clearly guarantees that the expected payoff is no greater than $1/4$ since $h < 2$ so the Searcher can dig to depth $1$ in at most one location.

So we just need to show that the Searcher's strategy guarantees he will win with probability at least $1/4$ against any pure Hider strategy. We will assume in this proof (and in the proofs of the optimality of the Searcher strategies in the other 3 lemmas) that the objects are hidden at depths $x$ and $y$ with $x \ge y$. We split the analysis into $4$ cases: (a) the two objects are in the same locations, (b) the objects are in different locations with $x \ge 5/6$, (c) the objects are in different locations with $x \in [1/2,5/6]$, and (d) the objects are in different locations with $x \le 1/2$.

\begin{description}
\item[Case (a): objects in the same location.]
In this case, it is clear that the Searcher wins if $L_1$ contains the objects, and this happens with probability $1/4$.
\item[Case (b): objects in different locations with $x \ge 5/6$.]
In this case there are $12$ equally probable combinations for the locations of the two objects, with respect to the Searcher's ordering. We show that the Searcher wins for the $3$ of these combinations in which the object at depth $x$ is in $L_1$. 

Indeed, if the other object is in $L_2$ then the Searcher will certainly find both of them, since he will dig to depth $1/2$ in $L_1$, then to depth $y$ in $L_2$, finding one of the objects, then to depth $x$ in $L_1$, finding the other (a total depth of only $x+y \le 1$). 

If the other object is in $L_3$ or $L_4$ then the Searcher will dig to depth $1/2$ in $L_1$ then to depth $1/2$ in $L_2$ before continuing to depth $1$ in $L_1$ and then to depth $1-x$ in $L_3$ and $L_4$, thereby finding both objects after digging to a total depth $d$ given by
\[
d = 1+1/2+2(1-x)=7/2-2x \le 11/6,
\]
with the last inequality following from $x \ge 5/6$.

\item[Case (c): objects in different locations with $x \in \lbrack 1/2, 5/6 \rbrack$.]
Similarly, in this case we show that the Searcher wins in $3$ of the $12$ possible combinations for the locations of the two objects. First suppose the object at depth $y$ is in $L_2$ and the object at depth $x$ is in $L_1$ or $L_3$. In this case, the Searcher digs to depth $1/2$ in $L_1$ then to depth $y$ in $L_2$, finally continuing to depth $1-y$ in $L_1$ and $L_3$, finding both objects after digging to a total depth $d$ given by
\[
d = (1-y)+y+(1-y)=2-y \le 11/6,
\]
since $y \ge 1/6$.

Finally, suppose the object at depth $x$ is in $L_2$ and the other is in $L_1$. Then the Searcher digs to depth $y$ in $L_1$, finds the first object, continues to depth $1$ in $L_1$ and then digs to depth $1-y$ in $L_2$, finding both objects after digging a maximum total depth of $d = 1+(1-y) \le 11/6$.

\item[Case (d): objects in different locations with $x \le 1/2$.]
In this case the Searcher wins in $4$ of the $12$ possible combinations: when the objects are in $L_1$ and $L_2$ or $L_2$ and $L_3$. Indeed, if they are in $L_1$ and $L_2$, the Searcher finds the object in $L_1$ before reaching depth $1/2$. He then continues digging to depth $1$ in $L_1$, and then digs in $L_2$, finding the other object having dug a total depth of no more than $3/2$.

If the objects are in $L_2$ and $L_3$ then the Searcher digs to depth $1/2$ in $L_1$, before digging in $L_2$ until finding an object at some depth $x \le 1/2$. He then digs to depth $1-x$ in $L_1$ and then digs in $L_3$ until finding the other object at some depth $y \le 1/2$, making a total depth $d$ satisfying
\[
d \le (1-x) + x + y \le 3/2 \le 11/6.
\]
\end{description}
\end{proof} 
Next we give the solution of the game in the case $h \in [11/5,7/3)$.

\begin{lemma}
\label{lemma:11/5-7/3}
For $n=4,k=2$ and $h \in [11/5,7/3)$, the value of the game is $9/20$. The Hider's optimal strategy is choose equiprobably between the $20$ pure strategies in $D(1/3) \cup E(1/3) \cup E(2/3)$.

The Searcher's optimal strategy is to number the locations randomly and proceed as follows.
\begin{enumerate}
\item[Stage 1:] Dig according to the sequence
\[
(3/5,0,0,0),(3/5,2/5,0,0),(4/5,3/5,0,0),(4/5,4/5,0,0),(1,4/5,0,0),(1,1,0,0).
\]
\item[Stage 2:] If an object is found in $L_1$, follow $IS(1234)$; if an object is found in $L_2$, with probability $4/5$ follow $IS(1234)$ and with probability $1/5$ follow $IS(134)$.
\end{enumerate}

\end{lemma}

\begin{proof}
We first show that the Hider's strategy ensures an expected payoff of no more than $9/20$ against any pure search strategy of the Searcher. We will show that whatever strategy the Searcher chooses, he can check at most $9$ of the Hider's $20$ possible configurations for the $2$ objects. First note that if we fix the Hider's strategy, we only need consider a finite number of Searcher strategies, where in every time interval $[t/3,(t+1)/3)$ for $t=0,1,2,\ldots$ he digs in the same location. Hence we only need to solve a finite optimization problem. Also note that after Stage 1 of the Searcher's strategy has been employed and he has found an object, it is clear that he should proceed to look for the remaining object using an intelligent search.

We analyze the problem by splitting it into $3$ cases, depending on Stage 1 of the Searcher's strategy up to time $1$. By symmetry we can assume that Searcher begins by digging in $L_1$, followed by the other $3$ locations in increasing order.

\begin{description}
\item[Case 1: Stage 1 of the Searcher's strategy begins with $(2/3,0,0,0),(2/3,1/3,0,0)$.]
In this case, there are two possibilities: either the Searcher has or has not found an object by time $1$. If he has not, then he can either continue digging in $L_2$ to depth $2/3$, in which case he wins against precisely $2$ of the Hider's pure strategies ($(0,2/3,1/3,0)$ and $(0,2/3,0,1/3)$); or he can dig to depth $1/3$ in $L_3$ or $L_4$, in which case he will only win against $1$ of the Hider's pure strategies.

If the Searcher has found an object by time $1$, it must be at time $1/3$, $2/3$ or $1$. If it is at time $1$ then the Searcher can win against $1$ further Hider pure strategy, say $(0,\{1/3,1\},0,0)$. If it is at time $2/3$ then the Searcher can win against each of the $4$ Hider pure strategies for which there is an object at depth $2/3$ in $L_1$. If it is at time $1/3$ then the Searcher can win against $2$ further Hider pure strategies, say $(\{1/3,1\},0,0,0)$ and $(1/3,2/3,0,0)$.

In total this means the Searcher wins against a maximum of $2+1+4+2=9$ of the $20$ pure strategies of the Hider.

\item[Case 2: Stage 1 of the Searcher's strategy begins with $(1/3,0,0,0),(1/3,1/3,0,0),(2/3,1/3,0,0)$.]
This is similar to the previous case, in that if the Searcher has not found an object by time $1$ he can win against a maximum of $2$ of the Hider's pure strategies. Also, if he finds an object at time $1/3$ then he wins against $2$ further pure strategies of the Hider. If he finds an object at time $1$ however, he can win against the $3$ Hider pure strategies of $(\{2/3,1\},0,0,0)$, $(2/3,0,1/3,0)$ and $(2/3,0,0,1/3)$; and if he finds an object at time $2/3$ he can win against only $2$ pure strategies of the Hider, say $(2/3,1/3,0,0)$ and $(0,\{1/3,1\},0,0)$. This also sums to $9$ pure Hider strategies in total.

\item[Case 3: Stage 1 of the Searcher's strategy begins with $(1/3,0,0,0),(1/3,1/3,0,0),(1/3,1/3,1/3,0)$.]
In this case, it is easy to check that if the Searcher has not found anything by time $1$, he can win against a maximum of $3$ of the Hider's pure strategies. If he finds an object at time $1/3$, $2/3$ or $1$ then he can win against at most $2$ of the Hider's pure strategies. So in total the Searcher wins against $3+2+2+2=9$ of the Hider's pure strategies.
\end{description} 

We must also show that the Searcher's strategy guarantees that he will win with probability at least $9/20$ against any pure Hider strategy where the objects are hidden at depths $x$ and $y$ with $x \ge y$. We sketch the proof and leave it to the reader to check the details.

First suppose the objects are buried in the same location. In this case, it is easy to check that the Searcher wins with probability $1$ if they are both in $L_1$ and he wins with probability $4/5$ if they are in $L_2$, so that the overall probability of a win is $(1+4/5)/4=9/20$.

Now suppose the objects are in different locations with $x,y \in [2/5,3/5]$. In this case we claim that the Searcher wins with probability at least $1/2$: in particular, he wins if the objects are in $L_1$ and $L_2$, $L_1$ and $L_3$, or $L_2$ and $L_3$. This is easily verified. 

If the objects are in different locations with $x \le 1/5$, then the Searcher wins against the same combinations of the locations as in the previous case.

Finally, suppose the objects are in different locations with $x \ge 1/5$ and $y \le 2/5$. We claim that the Searcher wins if the deepest object is in $L_1$ or if it is in $L_2$ and the other object is in $L_1$ or $L_3$. He also wins with probability $1/5$ if either the deeper object is in $L_2$ and the other is in $L_4$ or the deeper object is in $L_3$ and the other is in $L_2$. To check these claims it is best to split the analysis into the $4$ cases, $x \in [1/5,2/5]$, $x \in [2/5,3/5]$, $x \in [3/5,4/5]$ and $x \in [4/5,1]$. So the total probability with which he wins is $(5+2/5)/12 =9/20$.
\end{proof}

We note that for the case described in Lemma \ref{lemma:11/5-7/3}, the optimal Searcher strategy involves digging in two locations at the same time. None of the other Searcher strategies we present here have this property, and indeed it rarely occurs in optimal solutions to previous cases of the game.

We present the solution to the game for $h \in [7/4,9/5)$ in Lemma \ref{lemma:7/4-9/5} and for $h \in [9/5,11/6)$ in Lemma \ref{lemma:9/5-11/6} with sketch proofs.

\begin{lemma}
\label{lemma:7/4-9/5}
For $n=4,k=2$ and $h \in [7/4,9/5)$, the value of the game is $9/40$. The Hider's optimal strategy is to choose equiprobably between the $40$ pure strategies in $D(1/5) \cup D(2/5) \cup E(1/5) \cup E(2/5) \cup E(3/5) \cup E(4/5)$.

The Searcher's optimal strategy is to number the locations randomly and with probability $3/4$ proceed as follows.
\begin{enumerate}
\item[Stage 1:] Dig according to the sequence $(3/4,0,0,0),(3/4,1/4,0,0),(1,1/4,0,0),(1,3/4,0,0)$.
\item[Stage 2:] If an object is found in $L_1$, follow $IS(1234)$; if an object is found in $L_2$, follow $IS(134)$.
\end{enumerate}
With probability $1/4$, proceed as follows.
\begin{enumerate}
\item[Stage 1:] Dig according to the sequence $(3/4,0,0,0),(3/4,3/4,0,0),(1,3/4,0,0)$.
\item[Stage 2:] If an object is found in $L_1$, then with probability $3/5$ follow $IS(1234)$ and with probability $2/5$ follow $IS(234)$; if an object is found in $L_2$, follow $IS(134)$.
\end{enumerate}

\end{lemma}

\begin{proof}
The proof of the optimality of the Hider strategy is similar to the proof of the upper bound in Lemma \ref{lemma:11/5-7/3}, and we give an outline, leaving it to the reader to check the details.

We need to show that the Searcher can win against a maximum of $7$ of the Hider's pure strategies. We break the analysis down into $5$ cases, depending on the depth the Searcher has dug in each location after time $4/5$ in Stage 1 of his strategy. For each of these cases there are two sub-problems that must be solved: how many Hider pure strategies can the Searcher win against if he has not found an object by time $4/5$, and how many Hider pure strategies can the Searcher win against if he has found an object by time $4/5$? Denote these numbers by $N_1$ and $N_2$ and we give their values in each of the $5$ cases in Table \ref{tab:Lemma4-hider}, showing that they sum to no more than $9$.

\begin{table}[ht]
	\centering
\begin{tabular}{|c|c|c|c|}
	\hline
	Depths dug at time $4/5$ & &  & Max. number of \\ 
	in Stage 1 & $N_1$ & $ N_2$ & pure strategies beaten \\\hline
	$(4/5,0,0,0)$ & $0$ & $9$ & $9$ \\ 
	$(3/5,1/5,0,0)$ & $3$ & $6$ & $9$ \\ 
	$(2/5,2/5,0,0)$ & $3$ & $6$ & $9$\\
	$(2/5,1/5,1/5,0)$ & $4$ & $5$ & $9$\\
	$(1/5,1/5,1/5,1/5)$ & $4$ & $4$ & $8$\\
	\hline
\end{tabular}
	\caption{Performance of the optimal Hider strategy for $h \in [7/4,9/5)$.}
		\label{tab:Lemma4-hider}
\end{table}

To show that the Searcher strategy is optimal amounts to checking that it wins with probability at least $9/40$ against three classes of Hider strategy. (As usual we assume the objects are at depths $x$ and $y$ with $x \ge y$.) The first class is when both the objects are in the same location. In this case, the Searcher will win with probability $9/40$ if they are in $L_1$.

Next, if the objects are in different locations with $x \ge 3/4$, then the Searcher wins according to the probabilities given in Table \ref{tab:Lemma4-searcher1}, where the Hider strategies are written with respect to the Searcher's ordering of the locations.

\begin{table}[ht]
	\centering
\begin{tabular}{|c|c|c|c|c|}
	\hline
	Hider strategy & $(x,y,0,0)$ & $(y,x,0,0)$ & $(x,0,y,0)$ & $(x,0,0,y)$ \\ 
	\hline
	Minimum probability Searcher wins & $1$ & $ 1/4 \cdot 2/5 $ & $3/4+1/4 \cdot 2/5$ & $3/4$ \\
	\hline
\end{tabular}
	\caption{Performance of the optimal Searcher strategy for $h \in [7/4,9/5)$ when objects are in different locations at $x \ge 3/4$ and $y \le x$.}
		\label{tab:Lemma4-searcher1}
\end{table}

Since there are $12$ possible orderings, the total probability that the Searcher wins is $(1 + 1/4 \cdot 2/5 + 3/4+1/4 \cdot 2/5 + 3/4)/12 = 9/40$.

Finally, if the objects are at depths $x,y \le 3/4$, the Searcher wins according to the probabilities given in Table \ref{tab:Lemma4-searcher2}.

\begin{table}[ht]
	\centering
	\begin{tabular}{|c|c|c|c|c|}
	\hline
	Hider & $(x,y,0,0)$ or & $(x,0,y,0)$ or & $(0,x,y,0)$ or  \\
	strategy & $(y,x,0,0)$ & $(y,0,x,0)$ & $(0,y,x,0)$ \\
	\hline 
	Minimum probability Searcher wins & $1$  &  $1/10$ & $1/4$ \\
	\hline
	\end{tabular}
	\caption{Performance of the optimal Searcher strategy for $h \in [7/4,9/5)$ when objects are in different locations at $x,y \le 3/4$.}
		\label{tab:Lemma4-searcher2}
\end{table}

Hence the total probability of a Searcher win is $2(1+  1/10  + 1/4)/12 = 9/40$.

\end{proof}

\begin{lemma}
\label{lemma:9/5-11/6}
For $n=4,k=2$ and $h \in [9/5,11/6)$, the value of the game is $7/30$. The Hider's optimal strategy is to choose equiprobably between the $30$ pure strategies in $D(1/6) \cup D(1/2) \cup E(1/6) \cup E(1/2) \cup E(5/6)$.

The Searcher's optimal strategy is to number the locations randomly and with probability $2/3$ proceed as follows.
\begin{enumerate}
\item[Stage 1:] Dig according to the sequence $(1,0,0,0),(1,4/5,0,0)$.
\item[Stage 2:] If an object is found in $L_1$, follow $IS(1234)$; if an object is found in $L_2$, follow $IS(34)$.
\end{enumerate}
With probability $1/3$, proceed as follows.
\begin{enumerate}
\item[Stage 1:] Dig according to the sequence $(3/5,0,0,0),(3/5,3/5,0,0),(1,3/5,0,0),(1,4/5,0,0)$.
\item[Stage 2:] If an object is found in $L_1$, then with probability $4/5$ follow $IS(1234)$ and with probability $1/5$ follow $IS(234)$; if an object is found in $L_2$, follow $IS(134)$.
\end{enumerate}

\end{lemma}

\begin{proof}
We first show that the Searcher can win against a maximum of $9$ of the Hider's $30$ pure strategies. Similarly to the proof of Lemma \ref{lemma:7/4-9/5}, we break the analysis down into $4$ cases, depending on the depth to which the Searcher has dug in each location after time $5/6$ in Stage 1 of his strategy. The analysis is summed up in the table below, using the notation $N_1$ and $N_2$ as in the proof of Lemma \ref{lemma:7/4-9/5}.

\begin{table}[ht]
	\centering
	\begin{tabular}{|c|c|c|c|}
	\hline
	Depths dug at time $5/6$ & &  & Max. number of \\ 
	in Stage 1 & $N_1$ & $ N_2$ & pure strategies beaten \\\hline
	$(5/6,0,0,0)$ & $0$ & $7$ & $7$ \\ 
	$(1/2,1/3,0,0)$ & $3$ & $4$ & $7$ \\ 
	$(1/2,1/6,1/6,0)$ & $2$ & $5$ & $7$\\
	$(2/6,1/6,1/6,1/6)$ & $3$ & $4$ & $7$\\
	\hline
	\end{tabular}
	\caption{Performance of the optimal Hider strategy for $h \in [9/5,11/6)$.}
		\label{tab:Lemma5-hider}
\end{table}

To show that the Searcher strategy is optimal is also similar to the proof of Lemma \ref{lemma:7/4-9/5}. We check that the strategy wins with probability at least $7/30$ against $4$ classes of Hider strategy. As before, the first is when both the objects are in the same location and the Searcher wins with probability $7/30$ if they are in $L_1$.

The second is if the objects are in different locations with $x \ge 4/5$, then the Searcher wins according to the probabilities in Table \ref{tab:Lemma5-searcher1}.

\begin{table}[ht]
	\centering
	\begin{tabular}{|c|c|c|c|c|}
	\hline
	Hider strategy & $(x,y,0,0)$ & $(y,x,0,0)$ & $(x,0,y,0)$ & $(x,0,0,y)$ \\ 
	\hline
	Minimum probability Searcher wins & $1$ & $ 1/3 \cdot 1/5 $ & $1$ & $2/3+1/3 \cdot 1/5$ \\
	\hline
	\end{tabular}
	\caption{Performance of the optimal Searcher strategy for $h \in [9/5,11/6)$ when objects are in different locations at $x \ge 3/4, y \le x$.}
		\label{tab:Lemma5-searcher1}
\end{table}

So the total probability that the Searcher wins is  $(1 + 1/3 \cdot 1/5 + 1 + 2/3+1/3 \cdot 1/5)/12 = 7/30$.

Next, if the objects are in different locations with $x \in [3/5,4/5]$, the Searcher wins according to the probabilities in Table \ref{tab:Lemma5-searcher2}.

\begin{table}[ht]
	\centering
	\begin{tabular}{|c|c|c|c|c|c|}
	\hline
	Hider strategy & $(x,y,0,0)$ & $(y,x,0,0)$ & $(x,0,y,0)$ & $(y,0,x,0)$  & $(0,y,x,0)$\\
	\hline 
	Minimum probability Searcher wins & $1$ & $ 1 $ & $11/15$ & $1/15$ &  $1/3$ \\
	\hline
	\end{tabular}
	\caption{Performance of the optimal Searcher strategy for $h \in [9/5,11/6)$ when objects are in different locations at $x \in [3/5,4/5], y \le x$.}
		\label{tab:Lemma5-searcher2}
\end{table}

In this case, the probability of a Searcher win is strictly greater than $7/30$.

Lastly, suppose the objects' depths satisfy $x,y \le 3/5$. In this case, the Searcher wins according to the probabilities in Table \ref{tab:Lemma5-searcher3}.

\begin{table}[ht]
	\centering
	\begin{tabular}{|c|c|c|c|c|}
\hline
Hider & $(x,y,0,0)$ or& $(x,0,y,0)$ or & $(0,x,y,0)$ or\\
strategy & $(y,x,0,0)$ & $(y,0,x,0)$  & $(0,y,x,0)$\\
\hline 
Minimum probability Searcher wins & $1$ &  $1/15$ &  $1/3$ \\
\hline
	\end{tabular}
	\caption{Performance of the optimal Searcher strategy for $h \in [9/5,11/6)$ when objects are in different locations at $x,y \le 3/5$.}
		\label{tab:Lemma5-searcher3}
\end{table}

So the total probability the Searcher wins is $2(1+1/15+1/3)/12= 7/30$.

\end{proof}

\section{More general results}
\label{sec:general}

In this section we give some more general results for the game in the case of arbitrary $n$. The first result says that for $k=2$, if $h \le n/2$, the value of the game is asymptotically equal to $h/n$ as $n \rightarrow \infty$. A similar result is presented in Section 3.6.6 of \cite{OpDenKelder}, though our theorem uses a different Searcher strategy and our proof approach is different, being rather more geometrical in flavor. We believe that this could be a useful approach in future work on analyzing the game asymptotically.  

\begin{thm}
Suppose $k=2$ and $h \ge n/2$. Then the value $V(n,h)$ of the game is 
\[
V(n,h) = h/n + \epsilon(n),
\]
where $\epsilon(n) \rightarrow 0 $ as $n \rightarrow \infty$.
\end{thm}

\begin{proof}
Consider the Hider strategy of placing both objects at depth $1$ in a randomly chosen location. Clearly the Searcher cannot win with probability greater than $\lfloor h \rfloor /n$, so $V \le \lfloor h \rfloor /n$.

Now consider the Searcher strategy which digs up to depth $1$ in each location in a random order until the first object is found at depth $y$ (if found at all), after which he digs to depth $1-y$ in as many as possible of the remaining locations in a random order. We show that this strategy wins against any Hider strategy with probability at least $h/n-2/n$.

First suppose the two objects are hidden in the same location. Then the Searcher will be sure to find them both if he searches that location before running out of energy. This happens with probability $\lfloor h \rfloor /n$.

Now suppose the objects are in two different locations. We may as well assume the objects are at depths $y$ and $1-y$ for some $y \le 1/2$, since any other strategy is dominated. Suppose the object at depth $y$ is at the $i$th
location in the Searcher's random ordering and the object at depth $1-y$ is at the $j$th location. 

First suppose $i<j$. Then the Searcher wins if
\begin{align*}
i-1 + y + (1-y)(j-i) &\le h \mbox{ or}\\
iy + j(1-y) & \le h+1-y. 
\end{align*}
Similarly, if $i>j$ then the Searcher wins if
\[
iy + j(1-y) \le h+y.
\]
Hence the Searcher will certainly win if $i$ and $j$ satisfy 
\begin{equation}
iy + j(1-y)  \le h.
\label{eq:i-j}
\end{equation}
To calculate the probability $p$ that the Searcher wins amounts to finding the proportion of non-negative integer coordinates $(i,j)$ with $i,j \le n$ satisfying (\ref{eq:i-j}). This is approximately equal to the ratio of the area $A$ of the polygon $P$ in the positive quadrant of the $(i,j)$-plane bounded by the inequalities $i \le n$, $j \le n$ and (\ref{eq:i-j}), and $n^2$. More precisely, it is easy to see that $p \ge (A - 2n)/n^2 = A/n^2 - 2/n$, so it is sufficient to show that $A \ge nh$.

We split the analysis into the case that $h \le (1-y)n$ and $h > (1-y)n$. In the former case, the polygon $P$ is a trapezium and the area $A$ is given by
\begin{align*}
A &= \frac{n}{2}\left( \frac{h}{1-y} + \frac{h-ny}{1-y} \right)\\
&= n \left( \frac{h-ny/2}{1-y}\right)\\
& \ge n \left(\frac{h-hy}{1-y} \right) \mbox{ (since $h \ge n/2$)}\\
& = nh,
\end{align*}
as desired. 

In the case that $h > (1-y)n$, the polygon $P$ is a pentagon and we calculate $A$ by deducting from $n^2$ the area of the triangle formed by taking the complement of $P$:
\begin{align*}
A &= n^2 - \frac{1}{2}\left( n - \frac{h-ny}{1-y} \right) \left( n - \frac{h-n(1-y)}{y} \right)\\
&= n^2 - \frac{1}{2}  \left( \frac{(n-h)^2}{y(1-y)} \right)\\
& \ge n^2 - 2n(n-h)^2 \mbox{ (minimized at $y=1/2$)}\\
& = hn + (n-2h)(h-n)\\
&\ge hn \mbox{ (since $1/2 \le h \le 1$),}  
\end{align*}
as desired.

Hence we have
\[
h/n -2/n \le V \le \lfloor h \rfloor /n,
\]
and the theorem follows.
\end{proof}

Our other result in this section gives the exact value of the game for arbitrary $n$ and $k$ in the special case that $h < 1 + 1/k$. Before stating the result we define strategies for the players that will turn out to be optimal.

\begin{definition}
A \textbf{uniform allocation strategy} for the Hider is a placement of $k_i$ objects at depths $1/k,2/k,\ldots,k_i/k$ in each location $i$, for some choice of non-negative integers $k_1,\ldots,k_n$ summing to $k$. The \textbf{random allocation strategy} for the Hider is a uniformly random choice between all possible uniform allocation strategies.

A \textbf{uniform distribution strategy} for the Searcher is a strategy that digs in $L_i$ until finding $k_i$ objects (or running out of energy) for some choice of non-negative integers $k_1,\ldots,k_n$ summing to $k$. The \textbf{random distribution strategy} for the Searcher is a uniformly random choice between all possible random distribution strategies.
\end{definition}

These strategies are generalizations of the optimal strategies presented in Section \ref{sec:main} for $n=2,k=2$ and $h<3/2$, and the following proposition generalizes the solution of that case. It also generalizes 3.6.2 from \cite{OpDenKelder}.

\begin{proposition}
Suppose $h < 1 + 1/k$. The uniform allocation strategy is optimal for the Hider and the uniform distribution strategy is optimal for the Searcher. The value $V$ of the game is given by
\[
V=\frac{1}{\binom{n+k-1}{k}}.
\]
\end{proposition}

\begin{proof}
Let $N$ be the number of uniform allocation strategies, which is the same as the number of uniform distribution strategies. It is clear that the Searcher cannot check more than one of the Hider's uniform allocation strategies, so by using the uniform allocation strategy the Hider can ensure that the payoff is no more than $1/N$. On the other hand, however the objects are distributed, the Searcher will correctly guess the number of objects that are present in each location with probability $1/N$, so he can ensure the payoff is at least $1/N$. Hence the value is $1/N$.

It remains to show that $N=\binom{n+k-1}{k}$, that is the number of ways of choosing $n$ non-negative integers $k_1,\ldots,k_n$ that sum to $k$ is $\binom{n+k-1}{k}$. This is the problem of finding the number of \emph{weak compositions} of $k$ into exactly $n$ parts, which is well-known to be equal to $\binom{n+k-1}{k}$.
\end{proof}

\section{Conclusion and variants}

Our solution to the open problem posed in \cite{Fokkink-et-al} shows that optimal strategies can be complex and do not seem to follow any easily recognizable pattern. We expect that to find solutions to the game with $2$ objects hidden in $5$ or more locations will not give any further insight into the structure of the game, and we believe a better avenue for future research is to look at limit games, as in the current work of the first author \cite{Csoka}.

We note that in all four lemmas, the Hider's optimal strategy has the property described in the conjecture for caching games made in \cite{Fokkink-et-al} that after the first object is found, the remaining object is hidden optimally in the subgame. However, we know from \cite{Csoka} that the optimal Hider strategy does not necessarily have this property for large $n$ so that the conjecture is untrue in general. The conjecture has not been disproved in a variation of the game called the \emph{extremal game} in which the Hider must dig a total depth of exactly $1$. Since the Hider's optimal strategy has this property for all solved cases of the original game, the solution of the extremal game must be the same in these cases.

Another direction that future research could go in would be to look at a variation of the game where the Searcher cannot adapt his strategy but must specify in advance how deep he wishes to dig in each location. Some elementary analysis of this version of the game can be found in \cite{AFOL} and \cite{AFLC}, but we believe a lot more could be done in this area.

Finally, we mention two other possible extensions to the game, which were suggested by the Associate Editor. In both these extensions, it is possible that the Hider's optimal strategy has the structure conjectured in \cite{Fokkink-et-al}, and we leave these possibilities as questions for future research.

The first extension is a version in which the Searcher is restricted to searching the locations sequentially, so that once he has dug in so{\tiny }me location, he is not permitted to return to it. In many of the known solutions of the game the Searcher's behavior follows this requirement anyway, but in general the added restriction may decrease the value of the game, and a pattern may be more easily detectable. However it is not clear whether this version of the game would be any easier to analyze than the original version.

The second extension we mention is a discretized version of the game, for example when the Hider is restricted to placing the objects at depths $1/2$ or $1$. For $n=2$ the solution remains the same, but in general this restriction on the Hider will increase the value of the game, and hopefully simplify the analysis due to the Hider's strategy space being finite. If this problem is more tractable, it could be extended to a version in which the Hider can place objects at depths $1/k,2/k,\cdots,k/k$ for some fixed parameter $k$. Letting $k$ tend to $\infty$ could give some insight into the original problem.

\medspace

\textbf{Acknowledgement.} Endre Cs\'{o}ka was supported by the European Research Council (grant agreements no. 306493 and 648017).


\begin{thebibliography}{1}

\bibitem{Alpern-Fokkink} 
S. Alpern and R. Fokkink, Accumulation games on graphs, Networks 64 (2014), 40-47.

\bibitem{AFK}
S. Alpern, R. Fokkink, and K. Kikuta, On Ruckle's conjecture on accumulation games, SIAM J Control Optim 48 (2010), 5073-5083.

\bibitem{AFOL} 
S. Alpern, R. Fokkink, J. Op Den Kelder, and T. Lidbetter, Disperse or Unite? A Mathematical Model of Coordinated Attack, Lecture Notes in Comput. Sci. 6442, Springer, New York, 2010, 220-233.

\bibitem{AFLC} 
S. Alpern, R. Fokkink, T. Lidbetter, and N. S. Clayton, A Search Game model of the Scatter Hoarder's problem, J. R. Soc. Interface, 9 (2012), 869-879.

\bibitem{AFP}
S. Alpern, R. Fokkink, and C. Pelekis, A proof of the Kikuta–Ruckle Conjecture on cyclic caching of resources, J Optimiz Theory App, 153 (2012), 650-661.

\bibitem{Alpern-Gal}
S. Alpern and S. Gal, The Theory of Search Games and Rendezvous, Internat. Ser. Oper. Res. Management Sci. 319,
Kluwer Academic, Boston, 2003.

\bibitem{Avenhaus-et-al}
R. Avenhaus, M. Canty, D. M. Kilgour, B. Von Stengel, and S. Zamir, Inspection games in arms control, Eur J Oper Res, 90 (1996), 383-394.

\bibitem{Baston}
V. J. Baston, Some Cinderella Ruckle Type Games, in Search Theory: A Game Theoretic Perspective, S. Alpern, R. Fokkink,
L. Gasieniec, R. Lindelauf, and V. S. Subrahmanian, Springer-Verlag, New York, 2014, 85-104.

\bibitem{Baston-Bostock} 
V. J. Baston, and F. A. Bostock, A generalized inspection game, Nav Res Log 38 (1991), 171-182.

\bibitem{Benkosi-et-al}
S. J. Benkoski, M. G. Monticino, and J.R. Weisinger, A survey of the search theory literature, Nav Res Log 38 (1991), 469-494.

\bibitem{Csoka} E. Cs\'{o}ka, Limits of some combinatorial problems, arXiv preprint arXiv:1505.06984 (2015).

\bibitem{Ferguson-Melolikdakis}
T. Ferguson and C. Melolidakis, On the inspection game, Nav Res Log 45 (1998), 327-334.

\bibitem{Fokkink-et-al}
R. Fokkink, J. op den Kelder, and C. Pelekis, How to Poison Your Mother-in-Law, and Other Caching Problems, in Search Theory: A Game Theoretic Perspective, S. Alpern, R. Fokkink, L. Gasieniec, R. Lindelauf, and V. S. Subrahmanian, Springer-Verlag, New York, 2014, 85-104.

\bibitem{Garnaev}
A. Y. Garnaev, On a Ruckle problem in discrete games of ambush, Nav Res Log 44 (1997), 353-364.

\bibitem{Garnaev-book}
A. Garnaev, Search games and other applications of game theory, Vol. 485. Springer Science and Business Media, 2000.

\bibitem{Kikuta-Ruckle97} K. Kikuta and W. H. Ruckle, Accumulation games, Part 1: noisy search, J Optimiz Theory App 94
(1997), 395-408.

\bibitem{Kikuta-Ruckle2000} 
K. Kikuta and W. H. Ruckle, Continuous accumulation games in continuous regions, J Optimiz Theory App 106 (2000), 581-601.

\bibitem{Kikuta-Ruckle2002}
K. Kikuta and W. H. Ruckle, Continuous accumulation games on discrete locations, Nav Res Log, 49 (2002), 60-77.

\bibitem{OpDenKelder}
J. op den Kelder, Disperse or unite: a mathematical model for coordinated attack, TU Delft technical report, http://repository.tudelft.nl/, 2012.

\bibitem{Ruckle}
W. H. Ruckle, Geometric Games and Their Applications, Pitman, Boston, 1983.

\bibitem{Zoroa-et-al}
N. Zoroa, P. Zoroa, and J. Fern\'{a}ndez-S\'{á}ez, New results on a Ruckle problem in discrete games of ambush, Nav Res Log, 48 (2001), 98-106.

\end{thebibliography}
\end{document}